\documentclass[11pt]{amsart}

\usepackage[usenames,dvipsnames]{color}
\usepackage{array}
\usepackage{indentfirst,pict2e}

\usepackage{amsfonts}
\usepackage[leqno]{amsmath}
\usepackage{amsthm}
\usepackage{latexsym}

\usepackage[
	hyperindex,
	pagebackref,
	pdftex,
	pdfdisplaydoctitle,
	pdfpagemode=UseNone,
	breaklinks=true,
	extension=pdf,
	bookmarks=false,
	plainpages=false,
	colorlinks,
	linkcolor=linkblue,
	citecolor=linkblue,
	urlcolor=linkblue,
	pdfmenubar=true,
	pdftoolbar=true,
	pdfpagelabels,
	pdfpagelayout=SinglePage,
	pdfview=Fit,
	pdfstartview=Fit
]{hyperref}

\definecolor{linkblue}{rgb}{0,0.2,0.6}

\newcommand{\neturl}[1]{\href{#1}{{\sffamily{\texttt{#1}}}}}
\newcommand{\neturltilde}[2]{\href{#1}{{\sffamily{\texttt{#2}}}}}
\newcommand{\netarxiv}[1]{\href{http:arxiv.org/abs/#1}{{\sffamily{\texttt{arXiv:#1}}}}}
\usepackage[backrefs]{amsrefs}

\numberwithin{equation}{section}
\newtheorem{theorem}[equation]{Theorem}
\newtheorem{theorem*}{Theorem}
\newtheorem{problem*}{Problem}
\newtheorem{definition}[equation]{Definition}
\newtheorem{lemma}[equation]{Lemma}

\newtheorem{conjecture}[equation]{Conjecture}
\newtheorem{proposition}[equation]{Proposition}
\newtheorem{problem}[equation]{Problem}

\newcommand{\R}{\mathbb{R}}

\newcommand{\N}{\mathbb{N}}
\newcommand{\Q}{\mathbb{Q}}

\newcommand{\Pro}{\mathbb{P}}

\newcommand{\dblq}{/\!\!/}

\newcommand{\M}{\overline{M}} 

\newcommand{\rayl }{[}
\newcommand{\rayr}{]}

\newcommand{\Nef}{\operatorname{Nef}}
\newcommand{\Pic}{\operatorname{Pic}}

\newcommand{\sL}{\mathfrak{sl}}
\newcommand{\Sym}{\operatorname{Sym}}
\newcommand{\SL}{\operatorname{SL}}

\begin{document}

\pagenumbering{arabic}

\title{$\sL_2$ conformal block divisors and the nef cone of $\overline{M}_{0,n}$}

\author{David Swinarski}
\address{Department of Mathematics\\ University of Georgia\\ Athens, GA 30602}
\email{davids@math.uga.edu}


\date{\today}

\begin{abstract}
We show that $\sL_2$ conformal blocks do not cover the nef cone of
$\M_{0,6}$, or the $S_9$-invariant nef cone of $\M_{0,9}$.  A key
point is to relate the nonvanishing of intersection numbers between these
divisors and F-curves to the nonemptiness of some explicitly defined
polytopes.  Several experimental results and some open problems are also included.
\end{abstract}
\maketitle

\section{Introduction}

Vector bundles of conformal blocks on the moduli stack  of stable $n$-pointed genus $g$ curves $\overline{\mathcal{M}}_{g,n}$ were constructed in the 1980s by Tsuchiya, Ueno, and Yamada.  These vector bundles depend on three ingredients: a simple Lie algebra $\mathfrak{g}$, a nonnegative integer $\ell$ called the level, and an $n$-tuple of dominant integral weights $\vec{\lambda}$ in the Weyl alcove of level $\ell$ for $\mathfrak{g}$.  Fakhruddin's recent preprint \cite{Fakh} contains formulas for the Chern classes of these vector bundles and formulas for the intersection numbers of their determinant line bundles with certain curves (F-curves) in the moduli space.  This allows us to compute many new examples of conformal blocks for the first time.  Fakhruddin also shows that on $\M_{0,n}$ these vector bundles are globally generated, and hence their determinant line bundles, which we denote $D(\mathfrak{g},\ell,\vec{\lambda})$, are nef.  

A natural question to ask is whether every nef divisor on $\M_{0,n}$ arises this way.   For $n=4$ and $n=5$, one quickly finds that this is true (see Section \ref{exp section} below), but already for $n=6$, it is not so easy to find conformal block divisors covering $\Nef(\M_{0,n})$.  We will show in Section \ref{dont cover 6 section} that conformal blocks for $\mathfrak{g}=\sL_2$ do not cover $\Nef(\M_{0,6})$.  Unfortunately, this is not a question which can be settled with an exhaustive computational search; there are infinitely many Lie algebras and infinitely many levels to check.  

However, all the $S_6$-symmetric divisors on $\M_{0,6}$ are conformal blocks, and this suggests a second question: Does every $S_n$-symmetric nef divisor on $\M_{0,n}$ arise from conformal blocks?  Once again, one quickly finds that this is true for $n=4,5,6,7,8$ (see Section \ref{symm exp section} below).  However, for $n=9$, we have only found conformal blocks covering half the cone.

We elaborate on this now.  For $n=9$, the vector space of
$S_9$-symmetric divisors $\Pic(\M_{0,9})^{S_9} \otimes \R$ is 3-dimensional.  $\Nef(\M_{0,9}) \cap \Pic(\M_{0,9})^{S_9} $ is a polyhedral cone with 4 facets meeting along 4 extremal rays.  If we take a cross section of this cone, we obtain a planar quadrilateral as shown below:

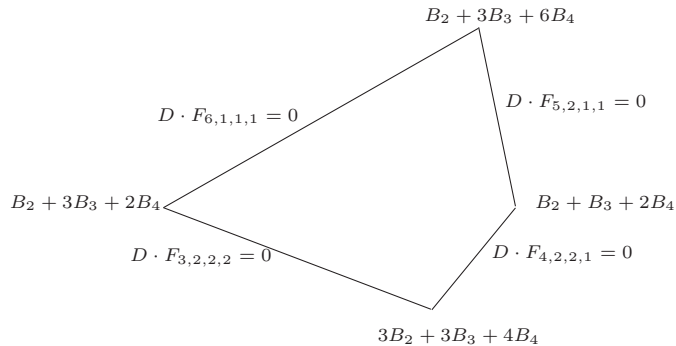
\begin{figure}[h]
\setlength{\unitlength}{.35in}
\begin{picture}(3,5)(-2,-2)
\put(1.15,0){ \tiny $B_2+B_3+2B_4$}
\put(-0.5,2.8){ \tiny $B_2+3B_3+6B_4$}
\put(-6.7,0){ \tiny $B_2+3B_3+2B_4 $}
\put(-1.2,-2.0){ \tiny $3B_2+3B_3+4B_4$}
\put(-4.26,0){\line(1,0.57){4.71}}
\put(-4.5,1.3){ \tiny $D \cdot F_{6,1,1,1} = 0$}
\put(-4.26,0){\line(1,-0.38){4.0}}
\put(-4.9,-0.8){ \tiny $D \cdot F_{3,2,2,2} = 0$}
\put(0.45,2.7){\line(1,-4.87){0.55}}
\put(0.7,1.5){ \tiny $D \cdot F_{5,2,1,1} = 0$}
\put(-0.25,-1.53){\line(1,1.22){1.25}}
\put(0.5,-0.75){ \tiny $D \cdot F_{4,2,2,1} =0$}
\end{picture}
\end{figure}

The leftmost three of these extremal rays are spanned by conformal block divisors.  
Unfortunately, as above, it is not possible to determine by a finite computer
search that the rightmost ray spanned by $B_2+B_3+2B_4$ is not a symmetrized
conformal block divisor, since there are infinitely many Lie algebras and
infinitely many levels that could be used.

More generally, a major open question is: 

\begin{problem*}
Given a nef divisor $D$ on $\M_{0,n}$, find a triple $(\mathfrak{g},\ell,\vec{\lambda})$ such that $\det \mathbb{V}(\mathfrak{g},\ell,\vec{\lambda}) = D$, or prove that no such triple exists.
\end{problem*}

In Section \ref{exp section} I will report some experimental results.  
In Section \ref{sl2 lemmas section}, I will give some lemmas concerning ranks and Chern classes of  $\sL_2$ conformal block
bundles. In Section \ref{dont cover 6 section}, I will show that conformal
block divisors for $\mathfrak{g}=\sL_2$ do not cover $\Nef(\M_{0,6})$. 
 In Section \ref{dont cover 9 section}, I will show that symmetrized conformal
block divisors for $\mathfrak{g}=\sL_2$ do not cover $\Nef(\M_{0,9}) \cap \Pic(\M_{0,n})^{S_9}$. 


\subsection*{Acknowledgements}
It is a pleasure to thank Valery Alexeev and Angela Gibney for many
useful conversations about vector bundles of conformal blocks.  Boris
Alexeev provided Lemma \ref{boris lemma} below, which turned out to be a crucial ingredient, and also wrote the first generation computer
programs that we used to compute examples of $\sL_2$ conformal block bundles before I wrote the more general
\texttt{Macaulay2} package \texttt{ConformalBlocks}.  Finally, I would like
to thank the UGA Conformal Blocks Seminar participants, including Maxim Arap, Brian
Boe, Bill Graham, Elham Izadi, Jim Stankewicz, and Robert Varley.   This work was partially financially supported by the University of Georgia's NSF VIGRE grant DMS-03040000. 

\subsection*{Software}  Computer calculations are essential to this paper.  I have written written
a package \texttt{ConformalBlocks} which can be used to compute ranks, divisor classes, and intersection numbers
of conformal block bundles and divisors in \texttt{Macaulay2} \cites{ConformalBlocks,Macaulay2}.  Before I implemented the fusion rules in \texttt{ConformalBlocks}, I used \texttt{KAC} for this purpose \cite{KAC}.  I 
also used the software \texttt{NefWiz} and \texttt{polymake} to
explore the subcones of the nef cone generated by conformal block divisors of different types \cites{NefWiz,
polymake}.

\section{Experimental results} \label{exp section}
For $\M_{0,4}$ and $\M_{0,5}$, the nef cone is covered by pullbacks from GIT quotients of the form $(\Pro^{1})^{n} \dblq_{L} \SL_2$; see \cite{AS}.  By \cite{Fakh}*{Theorem 4.5}, these are conformal block bundles.  

\subsection{Conformal blocks when $n=6$}
For $\M_{0,6}$ we have not been able to find a conformal blocks descriptions for every extremal rays of the nef cone.  The nef cone of $\M_{0,6}$ can be computed by modern software such as \texttt{polymake} (\cite{polymake}).  One first computes the cone of divisors which nonnegatively intersect the $F$-curves.  This gives an upper bound for the nef cone.  One can prove that the extremal rays of this cone are all nef, and so this upper bound cone is the nef cone.  $\Nef \M_{0,6}$ has 3190 extremal rays, and these fall into 28 $S_6$ orbits, a calculation first reported by Faber \cite{Faber}.

We have conformal block descriptions for 11 of these orbits.  Six orbits are spanned by line bundles pulled back from the GIT quotients $(\Pro^{1})^{6} \dblq_{L} \SL_2$ studied in \cite{AS}; by \cite{Fakh}*{Theorem 4.5}, these are conformal block descriptions either for $\sL_2$ and some level $\ell$ or $\sL_k$, level 1 for some $k$.  For $\sL_2,$ the divisor $D(\sL_2,1,(1,1,1,1,1,1))$ is also extremal.  The rays contained in these seven orbits are the only extremal rays hit by $\sL_2$ conformal blocks; see Theorem \ref{n6 Fakh cone equals sl2 cone} below.  If we allow higher rank Lie algebras, we find that $D(\sL_3,1,(\omega_1^3 \omega_2^3)) $ and $D(\sL_3,2,(2\omega_1,\omega_1,\omega_1,2\omega_2,\omega_2,\omega_2)) $ are extremal, covering two more orbits, and $D(\mathfrak{sl}_{6},2,(2\omega_1,2\omega_1,2\omega_3,2\omega_3,2\omega_5,2\omega_5))$ and ${D}(\mathfrak{sl}_{6},2,(2\omega_1,2\omega_1,2\omega_3,2\omega_4,2\omega_4,2\omega_5))$ are extremal, covering two more orbits. It is an open question whether higher rank Lie algebras or higher levels would yield the remaining orbits; alternatively, perhaps conformal block determinants do not cover the nef cone of $\M_{0,6}$ no matter what $\mathfrak{g}$ is used.

In the table below, we list the 28 orbits, along with a representative of each.  The representatives $D$ are given by a vector of length 16; these are the coordinates of a $D$ in the so-called nonadjacent basis, ordered as follows: $$\{\delta_{13},\delta_{14},\delta_{15},\delta_{24},\delta_{25},\delta_{26},\delta_{35},\delta_{36},\delta_{46}, \delta_{124},\delta_{125},\delta_{134},\delta_{135},\delta_{136},\delta_{145},\delta_{146}\}.$$
Patterns I and II are defined in Definition \ref{pattern I II def} below.  The stabilizer listed is for the representative of each orbit. If the group is not easily recognizable, I have given its GAP identifier, which is an ordered pair; the first coordinate is the order of the group.

All the orbits below except orbits 6 and 9 are spanned by big divisors.  To check this, I computed the top self-intersections of the representatives listed below.

\begin{tabular}{|l|l|l|l|l|l|}
\hline
\mbox{\small Orbit} & \mbox{\small Size} & \mbox{\small
  Representative} & \mbox{\small Pattern} & \mbox{\small Conformal
  block description} & \mbox{\small Stabilizer}\\
\hline
1 & 1 & \tiny $(1, 0, 1, 1, 0, 1, 1, 0, 1, 0, 0, 0, 2, 0, 0, 0)$ \normalsize & &  $D(\sL_2,1,(1,1,1,1,1,1))$& $S_6$ \\
\hline 
2 & 1 & \tiny $(0, 1, 0, 0, 1, 0, 0, 1, 0, 1, 1, 1, 0, 1, 1, 1)$ \normalsize& I & $D(\sL_2,2,(1,1,1,1,1,1))$ 
& $S_6$ \\
\hline 
3 & 6 &  \tiny $(0, 1, 0, 1, 0, 0, 1, 0, 1, 1, 0, 0, 1, 0, 0, 1)$ \normalsize & &$D(\sL_2,3,(3,1,1,1,1,1))$ & $S_5$ \\
\hline 
4 & 10 & \tiny $(0, 1, 0, 0, 1, 0, 0, 1, 0, 1, 1, 0, 0, 1, 1, 1)$ \normalsize& I & & $(72,40)$\\
\hline
5 & 10 &\tiny $(1, 0, 1, 1, 0, 1, 1, 0, 1, 0, 0, 0, 3, 0, 0, 0)$ \normalsize &I & $D(\sL_3,1,(\omega_1,\omega_1, \omega_1,  \omega_2,\omega_2,\omega_2)) $ & $(72,40)$ \\
\hline
6 & 15 &\tiny $(1, 0, 0, 1, 0, 0, 0, 0, 0, 0, 0, 0, 1, 1, 0, 0)$ \normalsize & &$D(\sL_2,1,(1,1,1,1,0,0))$ & $S_4 \times Z_2$ \\
\hline
7 & 15 & \tiny $(0, 2, 0, 0, 1, 0, 0, 1, 0, 1, 1, 1, 0, 1, 1, 1)$ \normalsize & I&$ D(\sL_2,3,(2,2,1,1,1,1)) $ &  $S_4 \times Z_2$\\
\hline
8 & 15 & \tiny $(1, 1, 0, 0, 1, 1, 1, 0, 1, 1, 1, 1, 2, 1, 1, 1)$ \normalsize & I & &$S_4 \times Z_2$\\
\hline
9 & 30 & \tiny $(0, 1, 0, 0, 0, 0, 1, 0, 1, 1, 0, 0, 1, 0, 0, 1)$ \normalsize & & $D(\sL_2,2,(2,1,1,1,1,0))$ &  $S_4$\\
\hline
10 & 45 & \tiny $(1, 0, 1, 1, 0, 1, 1, 0, 1, 0, 0, 0, 3, 1, 0, 0)$ \normalsize &I, II & $D(\mathfrak{sl}_{6},2,(2\omega_1,2\omega_1,2\omega_3,2\omega_3,2\omega_5,2\omega_5))$&  $(16,11)$\\
\hline
11 & 60 & \tiny $(0, 1, 0, 2, 0, 0, 1, 1, 0, 2, 0, 0, 1, 1, 1, 1)$ \normalsize & I &$ D(\sL_2,4,(3,2,2,1,1,1))$ &  $(12,4)$\\
\hline
12 & 60 & \tiny $(0, 1, 1, 1, 0, 1, 1, 1, 0, 1, 0, 0, 2, 1, 1, 0)$ \normalsize &I, II & $D(\mathfrak{sl}_{6},2,(2\omega_1,2\omega_1,2\omega_3,2\omega_4,2\omega_4,2\omega_5))$ & $(12,4)$ \\
\hline
13 & 60& \tiny $(0, 2, 0, 0, 1, 1, 1, 1, 0, 2, 0, 1, 1, 1, 2, 2)$ \normalsize & I, II& &$(12,4)$\\
\hline
14 & 60 & \tiny $(0, 2, 0, 0, 1, 1, 1, 1, 0, 2, 1, 1, 1, 1, 2, 1)$ \normalsize & I & &$(12,4)$\\
\hline
15 & 60 & \tiny $(1, 0, 1, 1, 0, 1, 1, 0, 1, 0, 0, 0, 3, 1, 1, 0)$ \normalsize & I, II && $(12,4)$\\
\hline
16 & 60 & \tiny $(1, 1, 0, 1, 1, 1, 1, 0, 1, 0, 0, 1, 2, 1, 0, 1)$ \normalsize & II && $(12,4)$\\
\hline
17 & 72 & \tiny $(0, 2, 0, 0, 1, 1, 1, 1, 0, 1, 0, 1, 1, 1, 1, 2)$ \normalsize & I, II& &$(10,1)$\\
\hline
18 & 90 & \tiny $(0, 2, 0, 2, 1, 1, 1, 1, 0, 2, 0, 1, 1, 1, 1, 1)$ \normalsize &I & $D(\sL_3,2,(2\omega_1,\omega_1,\omega_1,2\omega_2,\omega_2,\omega_2))$ &$(8,3)$  \\
\hline
19 & 90 &\tiny $(0, 2, 0, 0, 1, 1, 1, 1, 0, 2, 0, 1, 1, 1, 1, 1)$ \normalsize  & I, II&& $(8,3)$\\
\hline
20 & 90 & \tiny $(0, 1, 0, 1, 1, 0, 0, 1, 0, 1, 1, 0, 0, 1, 1, 1)$ \normalsize & I, II&& $(8,3)$\\
\hline
21 & 180 &  \tiny $(0, 2, 0, 0, 1, 1, 1, 1, 0, 1, 0, 1, 1, 1, 1, 1)$ \normalsize&  I, II& &$Z_2 \times Z_2$\\
\hline
22 & 180 & \tiny $(0, 1, 0, 1, 1, 0, 0, 2, 0, 1, 1, 0, 0, 2, 1, 1)$ \normalsize & I, II& &$Z_2 \times Z_2$\\
\hline
23 & 180 &\tiny $(0, 1, 1, 1, 1, 0, 0, 2, 0, 1, 1, 0, 1, 2, 2, 1)$ \normalsize & I, II&& $Z_2 \times Z_2$\\
\hline
24 & 360 &\tiny $(0, 1, 1, 2, 1, 0, 0, 2, 0, 2, 1, 0, 1, 2, 2, 1)$ \normalsize &I, II& &$S_2$ \\
\hline
25 & 360 & \tiny $(0, 1, 1, 1, 0, 1, 1, 1, 1, 1, 0, 0, 2, 1, 1, 1)$ \normalsize & I, II& &$S_2$ \\
\hline
26 & 360 & \tiny $(0, 2, 0, 1, 1, 1, 2, 0, 1, 1, 0, 1, 2, 0, 1, 2)$ \normalsize & II&  &$S_2$ \\
\hline
27 & 360 & \tiny $(0, 1, 1, 1, 1, 0, 1, 1, 1, 2, 1, 0, 2, 1, 0, 1)$ \normalsize &I, II&& $S_2$ \\
\hline
28 & 360 & \tiny $(0, 2, 0, 1, 1, 1, 1, 1, 0, 1, 1, 1, 1, 1, 2, 2)$ \normalsize &I, II&& $S_2$ \\
\hline
\end{tabular}
\renewcommand{\arraystretch}{1}




\subsection{Symmetric divisors when $n=6,7,8$} \label{symm exp section}

We now consider symmetric divisors on $\M_{0,n}$.  For small values of $n$, we can compute $\Nef(\M_{0,n})^{S_n}:=\Nef(\M_{0,n}) \cap \Pic(\M_{0,n}) ^{S_n}$.  Then, for $n\leq8$, for each extremal ray, we find a conformal block divisor spanning that ray.  The results are given in the table below.  Following \cite{KeelMcKernanContractible}, let $B_j = \sum_{|I|=j}^{\lfloor n/2 \rfloor} \delta_I$ (unless $j=n/2$, when we also insist $1 \in I$).  The divisor classes $\{B_j: 2\leq j\leq \lfloor n/2 \rfloor \}$ form a basis of $\Pic(\M_{0,n})^{S_n}$.  In the table below, the extremal rays of $\Nef(\M_{0,n})^{S_n}$ are given by their coefficients in the $B_j$ basis.  In the cases where $\vec{\lambda}$ is not $S_n$-symmetric, we take the symmetrization: $\Sym D := \sum_{\sigma \in S_n} \sigma D$.  

\begin{center}
\begin{tabular}{lll}
&Ray & Generator \\
 {n=6}  
&$\rayl 2:1 \rayr$& $D(\sL_2,1,1^6)$\\ 
&$\rayl1:3 \rayr$& $D(\sL_2,2,1^6)$\\
&&\\
 {n=7}  
&$\rayl1:1 \rayr$& $\Sym D(\sL_2,2,2^{5}1^{2})$  \\
&$\rayl1:3 \rayr$& $\Sym  D(\sL_2,3,1^{6}2)$\\
&&\\
 {n=8}  
&$\rayl3:2:4 \rayr$& $D(\sL_2,1,1^8)$\\
&$\rayl2:6:5 \rayr$& $D(\sL_2,2,1^8)$\\
&$\rayl1:3:6 \rayr$& $D(\sL_2,3,1^8)$\\
&$\rayl6:11:8 \rayr$& $\Sym D(\sL_3,1,\omega_{2}^{7}\omega_{1})$\\
&&\\
 {n=9} 
&$\rayl3:3:4 \rayr$& $\Sym D(\sL_2,2,2^{7}1^{2})$ \\
&$\rayl1:3:6 \rayr$& $\Sym D(\sL_2,4,1^82)$ \\
&$\rayl1:3:2 \rayr$& $D(\sL_3,1,\omega_{1}^{9})$\\
&$\rayl1:1:2 \rayr$ & {???}\\
\end{tabular}
\end{center}

\subsection{Symmetric divisors when $n=9$}
We have done exhaustive searches for conformal block determinants for $\sL_{2}$, levels 1 through 12, and for $\sL_{k}$ level 1, $k=2\ldots, 18$.  We have also computed many sporadic examples for higher levels and other simple Lie algebras.  Some typical results are illustrated in the first three figures below and/or on the next page.  As above, we show a cross section of the symmetric nef cone.  Red dots indicate rays which are spanned by symmetrized conformal block determinants.  The conformal block subcone is outlined in blue.   The fourth figure shows cones of pullbacks of symmetric nef divisors from $\M_{0,8}$ and $\M_{0,7}$ for reference.  

\input{4pics}

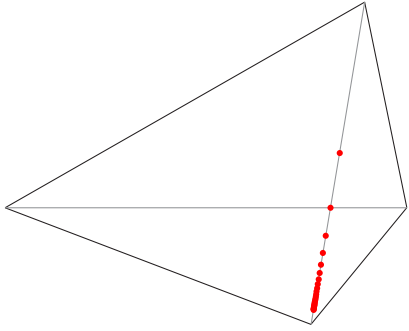
We have not found any examples which are multiples of $B_2+B_3+2B_4$, or any examples that lie to the right of the hyperplane joining the rays spanned by $3B_2+3B_3+4B_4 $ and $B_2+3B_3+6B_4$, which is given by the equation $D \cdot F_{6,1,1,1} = 3 D \cdot F_{5,2,1,1}$.  The closest we have found are the divisors $\Sym D(\mathfrak{so}_{4k+2}, 1,(\omega_{2k}^{8}, 0))$, which are symmetrizations of conformal block determinants for root system type $D_{2k+1}$.  In the third graphic (below) these are seen to lie on the hyperplane $D \cdot F_{6,1,1,1} = 3 D \cdot F_{5,2,1,1}$.  This leads to a natural question:

\begin{conjecture} \label{dontcoverconjecture}  Let $\mathfrak{g}$ be a simple Lie algebra.  Let $D:=\Sym D(\mathfrak{g},\ell,\vec{\lambda})$ be a symmetrized conformal block divisor on $\M_{0,9}$ associated to $\mathfrak{g}$.  Then $D$ satisfies the inequality $D \cdot F_{6,1,1,1} \leq  3 D \cdot F_{5,2,1,1}$.  In particular, conformal block divisors do not cover the symmetric nef cone of $\M_{0,9}$.
\end{conjecture}

I will not attempt to prove Conjecture \ref{dontcoverconjecture} here, and instead pursue an easier result.  In Theorem \ref{no 112}, I show that the eastern ray $E:=B_2+B_3+2B_4$ is not spanned by any symmetrized conformal block divisor for $\sL_2$.  This is evidence for Conjecture \ref{dontcoverconjecture}, but my method is simpler; proving the conjecture would seem to require computing or bounding these intersection numbers, whereas in the proof, I only prove that certain intersection numbers are nonzero (see Proposition \ref{main prop}).  At present, working with $\sL_2$ is substantially easier than with higher rank Lie algebras because we have several useful combinatorial expressions for ranks and first Chern classes of conformal block bundles whose generalization to higher rank are not yet known.

\section{When \texorpdfstring{$\mathfrak{g}=\sL_2$}{g=sl_2}} \label{sl2 lemmas section}

When $\mathfrak{g}=\sL_2$ many of the formulas for Chern classes of conformal blocks have nice combinatorial expressions.  We recall these now.  

First, we recall the fusion rules (ranks of conformal block bundles on $\M_{0,3}$) for $\sL_2$, which are well-known (see for instance \cite{Beauville}*{Lemma 4.2, Cor. 4.4}).
\begin{proposition} \label{fusion rules} Let $\mathfrak{g}=\sL_2$.
\begin{enumerate}
\item For $n=1$, 
\begin{displaymath} r_{\lambda} = \left\{
\begin{array}{l}
1 \quad \mbox{ if $\lambda = 0$}\\
0 \quad \mbox{ otherwise.}\\
\end{array} \right.
\end{displaymath}
\item For $n=2$, 
\begin{displaymath} r_{\lambda \mu} = \left\{
\begin{array}{l}
1 \quad \mbox{ if $\mu = \lambda$}\\
0 \quad \mbox{ otherwise.}\\
\end{array}\right.
\end{displaymath}
\item For $n=3$,
\begin{displaymath} r_{\lambda_{1} \lambda_{2} \lambda_{3}} = \left\{
\begin{array}{l}
1 \quad \mbox{ if $\Lambda \equiv 0 \bmod 2$ and $|\lambda_1 -\lambda_2 | \leq \lambda_{3} \leq \min\{\lambda_{1}+ \lambda_{2}, 2\ell - \lambda_{1}-\lambda_{2} \} $}\\
0 \quad \mbox{ otherwise.}\\
\end{array} \right.
\end{displaymath}
\end{enumerate}
\end{proposition}

In the formula below, $*$ denotes the involution on weights given by the longest word in the Weyl group $W(\mathfrak{g})$.  The involution $*$ is trivial for $\sL_2$, but we will sometimes write the $*$ in the sequel anyway, out of habit.

For any simple Lie algebra $\mathfrak{g}$, factorization readily yields the following formula for intersection numbers with F-curves (see \cite{Fakh}*{Proposition 2.5}): 
\begin{equation} \label{Fakh Fcurve formula} D(\mathfrak{g},\ell,\vec{\lambda}) \cdot F_{I_1,I_2,I_3,I_4} = \sum_{\vec{\mu} \in P_{\ell}^{4}} \deg \mathbb{V}_{\vec{\mu}} \, r_{\lambda_{I_1} \mu_{1}^{*}}  r_{\lambda_{I_2} \mu_{2}^{*}}  r_{\lambda_{I_3} \mu_{3}^{*}}  r_{\lambda_{I_4} \mu_{4}^{*}}.  
\end{equation}
Here  $\lambda_{I_j} \mu_j^{*}$ denotes the set of weights $\{ \lambda_i : i \in I_j\} \cup \{ \mu_j^{*} \}$.   We see that  to use (\ref{Fakh Fcurve formula}), one must be able to compute the degrees of conformal block bundles on $\M_{0,4} \cong \Pro^1$.  Fakhruddin has a general formula \cite{Fakh}*{Cor. 3.5} for this purpose, valid for any $\mathfrak{g}$, as well as a simpler formula for $\sL_2$ \cite{Fakh}*{Prop. 4.2}.  However,  when $\mathfrak{g}=\sL_2$, Boris Alexeev has found the following elegant formula:
\begin{lemma}[\cite{A}] \label{boris lemma} Let $\mathfrak{g}=\sL_2$ and let $n=4$.  Then $r_{abcd}=0$ if $a+b+c+d$ is odd, and if $a+b+c+d$ is even, 
\begin{equation} \label{r4} r_{abcd} = \max \left\{ 0,  1 + \frac{\ell}{2} - \frac{1}{4}\left(|a \! +\! b \! +\! c \! + \!  d \! - \!  2\ell| + |a \! + \!  b \! - \!  c \! - \!  d| + |\!   a \!  +\!   c \! - \!  b \! - \!  d| +|a \! + \!  d \! - \!  b \! - \!  c|  \right) \right\}.
\end{equation}
The degree of this vector bundle is
\begin{equation} \label{deg4} \deg \mathbb{V}_{abcd} = r_{abcd} \cdot \max\left\{0, \frac{1}{2}\left(a+b+c+d-2\ell \right) \right\}.
\end{equation}
\end{lemma}

Rasmussen and Walton give necessary and sufficient conditions for the rank of an $\sL_2$ conformal block bundle to be nonzero:
\begin{lemma}[\cite{RasmussenWalton}] \label{nonzero rank lemma} 
Let $\mathfrak{g}=\sL_2$.  Then $r_{\vec{\lambda}} \neq 0$ if and only $\Lambda$ is even, and for any subset $I \subseteq \{ 1,\ldots,n\}$ with $n-|I|$ odd, the inequality
\begin{equation}\label{gen tri ineq} 
\sum_{i \not\in I} \lambda_{i} - (n-|I|-1) \ell \leq  \sum_{i \in I} \lambda_{i}
\end{equation}
is satisfied.
\end{lemma}
This result is their system of inequalities (17) translated into my notation.  

The previous results lead to the following definition and proposition:

\begin{definition} \label{polytope Q def}
Let $\ell$ be a level, let $\vec{\lambda}$ an $n$-tuple of dominant integral weights for $\sL_2$ of level $\ell$, and let $F_{I_1,I_2,I_3,I_4}$ be an F-curve class on $\M_{0,n}$.  Let $\R^4$ have coordinates $\mu_1,\ldots,\mu_4$.  Then we define a polytope $\mathcal{Q} = \mathcal{Q}(\ell, \vec{\lambda} , F_{I_1,I_2,I_3,I_4}) \subset \R^4$ by the following inequalities: 
\begin{enumerate}
\item $0 \leq \mu_j \leq \ell$ for $j=1,\ldots,4$;
\item the Rasmussen-Walton inequalities associated to $ r_{\lambda_{I_j} \mu_{j}^{*}} \neq 0$, for $j=1,\ldots,4$;
\item the Rasmussen-Walton inequalities associated to $r_{\vec{\mu}} \neq 0$;
\item $\mu_1+\mu_2+\mu_3+\mu_4 \geq 2\ell+2$.
\end{enumerate}
\end{definition}

\begin{proposition} \label{polytope Q}
Let $\ell$ be a nonnegative integer, let $\vec{\lambda}$ be an $n$-tuple of dominant integral weights for $\sL_2$ of level $\ell$, and let ${I_1,I_2,I_3,I_4}$ be a partition of $\{1,\ldots,n\}$ into four nonempty subset.  Assume $r_{\vec{\lambda}} > 0$.  Then $D(\sL_2,\ell,\vec{\lambda}) \cdot F_{I_1,I_2,I_3,I_4} >0$ if and only if the polytope  $\mathcal{Q}(\ell, \vec{\lambda}, I_1,I_2,I_3,I_4)$ contains a integral point $\vec{\mu}$ whose coordinates have parities prescribed by $\mu_j \equiv \sum_{p \in I_j} \lambda_p \pmod{2}$. 
\end{proposition}
\begin{proof}
We combine (\ref{Fakh Fcurve formula}) with  Lemmas \ref{boris lemma} and \ref{nonzero rank lemma} above.  Since ranks and degrees of conformal block bundles on $\M_{0,n}$ are nonnegative integers, the right hand side of Equation (\ref{Fakh Fcurve formula}) is positive if it contains one nonzero summand.  We seek $\vec{\mu} = (\mu_1,\ldots,\mu_4)$ such that 
$$\deg \mathbb{V}_{\vec{\mu}} \, r_{\lambda_{I_1} \mu_{1}^{*}}  r_{\lambda_{I_2} \mu_{2}^{*}}  r_{\lambda_{I_3} \mu_{3}^{*}}  r_{\lambda_{I_4} \mu_{4}^{*}} \neq 0.$$
Thus we need $r_{\lambda_{I_j} \mu_{j}^{*}} \neq 0$ for $j=1,\ldots,4$.  This imposes the parity conditions $\mu_j \equiv \sum_{p \in I_j} \lambda_p \pmod{2}$ for $j=1,\ldots,4$, and, applying Lemma \ref{nonzero rank lemma} four times, four sets of Rasmussen-Walton inequalities.  We also require $\deg \mathbb{V}_{\vec{\mu}} \neq 0$, which by Lemma \ref{boris lemma} requires $r_{\mu} \neq 0$ and $\mu_1+\mu_2+\mu_3+\mu_4 > 2\ell$.  For   $r_{\mu} \neq 0$, the parity condition $\mu_1+\mu_2+\mu_3+\mu_4 \equiv 0 \pmod{2}$ is already satisfied, as $\mu_1+\mu_2+\mu_3+\mu_4 \equiv \Lambda \equiv 0 \pmod{2}$.  Thus it remains only to  impose the Rasmussen-Walton inequalities for $r_{\mu} \neq 0$.  Finally, since $\mu_1+\mu_2+\mu_3+\mu_4$ is even, we may replace the strict inequality $\mu_1+\mu_2+\mu_3+\mu_4 > 2\ell$ by the inequality $\mu_1+\mu_2+\mu_3+\mu_4 \geq 2\ell+2$.  

We have obtained exactly the inequalities which were used to define $\mathcal{Q}(\ell, \vec{\lambda}, I_1,I_2,I_3,I_4)$, and the parity condition in the statement of the proposition.
\end{proof}

In Sections \ref{dont cover 6 section} and \ref{dont cover 9 section} below, we will use the polytopes $\mathcal{Q}$ quite extensively.  Somewhat surprisingly, the calculations for specific F-curves when $n=6$ and $n=9$ show \textit{a posteriori} that the parity condition can be dropped.




\section{$\sL_2$ conformal blocks don't cover the nef cone for $n=6$} \label{dont cover 6 section}
In \cite{Fakh}*{Section 6}, Fakhruddin makes the following conjecture:  

\begin{definition}  We call an extremal ray of $\Nef(\M_{0,6})$ a \emph{Fakhruddin ray} if it is spanned either by $D(\sL_2,1,(1,1,1,1,1,1))$ or by one of the 127 pullbacks from GIT quotients $(\Pro^1)^n \dblq_{L} \SL_2$ studied in \cite{AS}.  We call an $S_6$ orbit of extremal rays a \emph{Fakhruddin orbit} if it is an orbit of Fakhruddin rays.  We call the subcone of $\Nef(\M_{0,6})$ spanned by Fakhruddin rays the \emph{Fakhruddin cone}.
\end{definition}

\begin{conjecture}[Fakhruddin Conjecture]  The cone generated by all $\sL_2$ conformal block divisors for all levels is equal to the Fakhruddin cone.
\end{conjecture}

Fakhruddin based his conjecture on calculations for small values of the level $\ell$.   In particular, his conjecture implies that the cone of all $\sL_2$ conformal block divisors is finitely generated (i.e. a polyhedral cone, not rounded).  To my knowledge, this is not known for any Lie algebra when $n \geq 6$.  However,  for a similar result, see \cite{GiansiracusaGibney}, where Giansiracusa and Gibney prove that for a fixed $n$, the cone generated by all $\sL_k$ level 1 conformal block divisors for all $k \geq 2$ is finitely generated.

I will not establish Fakhruddin's full conjecture here, but I will prove a partial result in this direction.  In this section, I will show that the non-Fakhruddin extremal rays are not spanned by $\sL_2$ conformal block divisors.  This establishes Fakhruddin's conjecture applied to extremal rays.  However, \textit{a priori}, there could be additional $\sL_2$ conformal blocks that are outside the Fakhruddin cone, but not extremal in $\Nef(\M_{0,6})$; Fakhruddin conjectures that this does not happen, but my methods give no information on this part of his conjecture.  

\begin{lemma} Let $E$ be a non-Fakhruddin extremal ray of $\Nef(\M_{0,6})$.  In seeking to find weights $\vec{\lambda}$ such that $E = D(\sL_2,\ell,\vec{\lambda})$, we may assume that $\lambda_i \geq 1$ for all $i$.
\end{lemma}
\begin{proof}
If one or more weights $\lambda_i \in \vec{\lambda}$ is zero, then $D(\sL_2,\ell,\vec{\lambda})$ is a pullback from $\M_{0,n'}$ with $n'<n$ (this is ``propagation of conformal blocks''), and hence not big.  However, only two of the 28 $S_6$ orbits of extremal rays are not big, and they are pullbacks from GIT quotients, hence already in the Fakhruddin cone.
\end{proof}





\begin{lemma} \label{n6 3111 lemma}
Suppose that $\vec{\lambda}$ satisfies $\lambda_i \geq 1 $ for all $i=1,\ldots, 6$, and 
\begin{enumerate}
\item[(I-1)] $2 + \lambda_5 + \lambda_6 \leq \lambda_1 + \lambda_2+\lambda_3 + \lambda_4$,
\item[(I-2)] $2 + \lambda_4 + \lambda_6 \leq \lambda_1 + \lambda_2+\lambda_3 + \lambda_5$,
\item[(I-3)] $2 + \lambda_4 + \lambda_5 \leq \lambda_1 + \lambda_2+\lambda_3 + \lambda_6$,
\item[(I-4)] $2\ell+2 \leq \lambda_1 + \lambda_2+\lambda_3 +\lambda_4 + \lambda_5+\lambda_6 $,
\item[(I-5)] $2 + \ell \leq \lambda_1+\lambda_2+\lambda_3$.
\end{enumerate}
Then $D(\vec{\lambda}) \cdot F_{\{4,5,6\},\{1\},\{2\},\{3\}} > 0$.
\end{lemma}
\begin{proof}
We use Proposition \ref{polytope Q}.  

Let $\mathcal{Q} = \mathcal{Q}(\ell,\vec{\lambda},\{4,5,6\},\{1\},\{2\},\{3\})$ be the polytope defined in Definition \ref{polytope Q def}.  Since the sets $I_2$, $I_3$, and $I_4$ in the partition are singletons, the Rasmussen-Walton inequalities associated $r_{I_j \mu_j^{*}} \neq 0$ are just the fusion rules (see Prop. \ref{fusion rules}) for $j=2,3,4$.  In particular, to have $r_{\lambda_{1} \mu_{2}^{*}}\neq 0$, we must have $\mu_{2}^{*} = \lambda_1 = \mu_{2}$, and similarly $\mu_{3}^{*} = \lambda_2 = \mu_{3}$ and $\mu_{4}^{*} = \lambda_3 = \mu_{3}$.

Write $\mu = \mu_1$.  Thus, we only need to prove that $\mathcal{Q}$ contains a point of the form $(\mu,\lambda_1,\lambda_2,\lambda_3)$ where $\mu \equiv \lambda_4+\lambda_5+\lambda_6 \pmod{2}$.  For this, we note that almost all the inequalities defining $\mathcal{Q}$ give the correct parity for $\mu$ when made equalities.  The exceptions are the conditions $0 \leq \mu \leq \ell$.  Thus, by adjusting the upper and lower bounds for $\mu$, we can arrange that if $\mathcal{Q}$ is
nonempty, it contains a point with the desired parity.  Specifically, if $\lambda_1+\lambda_2+\lambda_3 \equiv 0 \pmod{2}$ we require $\mu \geq 2$ else $\mu \geq 1$.  If $\lambda_1+\lambda_2+\lambda_3 \equiv \ell \pmod{2}$ we require $\mu \leq \ell$ else $\mu \leq \ell-1$.  Let us write  $\mathcal{Q}'$ for this adjusted polytope.
  
The inequalities (I-1) through (I-5) above are obtained from the inequalities defining
$\mathcal{Q}'$ by applying Fourier-Motzkin elimination to get rid of
$\mu$ and then discarding redundant inequalities.  \footnote{There exist software packages  that in principle can carry out Fourier-Motzkin elimination
  symbolically, for instance, \texttt{QEPCAD B}.  However, I have not been able to coax it to do this calculation for me yet.  It seems the number of variables and number of inequalities is too large.}

I have posted my notes for this calculation on my website: \\ \neturltilde{http://www.math.uga.edu/~davids/dontcover/}{http://www.math.uga.edu/$\sim$davids/dontcover/}. 
\end{proof}

\begin{lemma} \label{n6 2211 lemma}
Suppose $\ell \geq 3$ and that $\vec{\lambda}$ satisfies $\lambda_i \geq 1 $ for all $i=1,\ldots, 6$.  Then  $D(\vec{\lambda}) \cdot F_{\{1,2\},\{3,4\},\{5\},\{6\}} > 0$ if and only if the following 10 inequalities are satisfied:
\begin{enumerate}
\item[(J-1)] $2\ell+2 \leq \lambda_1 + \lambda_2+\lambda_3 +\lambda_4 + \lambda_5+\lambda_6 $
\item[(J-2)] $\lambda_1+\lambda_2 \leq 2\ell-1$
\item[(J-3)] $\lambda_3+\lambda_4 \leq 2\ell-1$
\item[(J-4)] $2 + \lambda_1 + \lambda_2 \leq \lambda_3 + \lambda_4+\lambda_5 + \lambda_6$
\item[(J-5)] $2 + \lambda_3 + \lambda_4 \leq \lambda_1 + \lambda_2+\lambda_5 + \lambda_6$
\item[(J-6)] $\lambda_1 + \lambda_2 + \lambda_3 +\lambda_4 \leq 2 \ell-2 + \lambda_5 + \lambda_6$
\item[(J-7)] $\lambda_1 + \lambda_2+\lambda_5 + \lambda_6 \geq \ell +2$
\item[(J-8)] $\lambda_3 + \lambda_4+\lambda_5 + \lambda_6 \geq \ell +2$
\item[(J-9)] $\lambda_1 + \lambda_2 \leq \ell - 2 + \lambda_5 + \lambda_6$
\item[(J-10)] $\lambda_3 + \lambda_4 \leq \ell - 2 + \lambda_5 + \lambda_6$
\end{enumerate}
\end{lemma}
\begin{proof}
The proof is analogous to the proof of Lemma \ref{n6 3111 lemma}.  I have posted my notes for this calculation on my website:  \neturltilde{http://www.math.uga.edu/~davids/dontcover/}{http://www.math.uga.edu/$\sim$davids/dontcover/}. 
\end{proof}

\begin{definition} \label{pattern I II def} Let $I_1 \subset \{1,2,3,4,5,6\}$.  We write $D \cdot F_{I_1} > 0$ if there exists some partition $I_1 \amalg I_2 \amalg I_3 \amalg I_4= \{1,2,3,4,5,6\}$ such that $D \cdot F_{I_1,I_2,I_3,I_4} > 0$.

 We say that an extremal ray $E$ of $\Nef(\M_{0,6})$ has \emph{Pattern I} if there exists $\sigma \in S_6 $ such that
\begin{equation}
\begin{array}{cccl}
\sigma E \cdot F_{\{1,2\}} & > &  0,\\
\sigma E \cdot F_{\{1,3\}} & > &  0,\\
\sigma E \cdot F_{\{2,3\}} & > &  0,\\
\sigma E \cdot F_{\{4,5\}} & > &  0,\\
\sigma E \cdot F_{\{4,6\}} & > &  0,\\
\sigma E \cdot F_{\{5,6\}} & > &  0,& \\
\sigma E \cdot F_{\{1,2,3\},\{4\},\{5\},\{6\}} & = &  0,\\
\sigma E \cdot F_{\{4,5,6\},\{1\},\{2\},\{3\}} & = &  0.
\end{array}
\end{equation}

We say that an extremal ray $E$ of $\Nef(\M_{0,6})$ has \emph{Pattern II} if there exists $\sigma \in S_6 $ such that
\begin{equation}
\begin{array}{cccl}
\sigma E \cdot F_{\{1,3\},\{2,4\},\{5\},\{6\}} & > &  0,\\
\sigma E \cdot F_{\{1,5\},\{3,4\},\{5\},\{6\}} & > &  0,\\
\sigma E \cdot F_{\{1,2,3\},\{4\},\{5\},\{6\}} & > &  0,\\
\sigma E \cdot F_{\{2,3,4\},\{1\},\{5\},\{6\}} & = &  0,\\
\sigma E \cdot F_{\{1,2\},\{3,4\},\{5\},\{6\}} & = &  0.
\end{array}
\end{equation}

\end{definition}

\begin{lemma}  \label{n6 ray structure}  Every non-Fakhruddin extremal ray $E$ has at least one of the two patterns I or II defined above.
\end{lemma}
\begin{proof}
This can be checked by computer.  I have posted my notes for this calculation on my website:  \neturltilde{http://www.math.uga.edu/~davids/dontcover/}{http://www.math.uga.edu/$\sim$davids/dontcover/}. 
\end{proof}

\begin{theorem}  \label{n6 Fakh cone equals sl2 cone}
Let $E$ be a non-Fakhruddin extremal ray (that is, an extremal ray of $\Nef(\M_{0,6})$ which is not a pullback from $(\Pro^1)^n \dblq_{L} \SL_2$ or $D(\sL_2, 1, (1,1,1,1,1,1))$).  Then $E$ is not a multiple of any $\sL_2$ conformal block divisor. 
\end{theorem}
\begin{proof}
By Lemma \ref{n6 ray structure}, we know that $E$ has at least one of the the two patterns I or II defined above.  We will show that $\sL_2$ conformal blocks cannot have either pattern. 

We consider Pattern I first.  We will show that if $D:= D(\sL_2, \ell, \vec{\lambda})$ satisfies 
\begin{displaymath}
\begin{array}{cccl}
D \cdot F_{\{1,2\}} & > &  0,\\
D \cdot F_{\{1,3\}} & > &  0,\\
D \cdot F_{\{2,3\}} & > &  0,\\
D \cdot F_{\{4,5\}} & > &  0,\\
D \cdot F_{\{4,6\}} & > &  0,\\
D \cdot F_{\{5,6\}} & > &  0, 
\end{array}
\end{displaymath}
then at least one of the intersection numbers $D \cdot F_{\{4,5,6\},\{1\},\{2\},\{3\}} $ or $D \cdot F_{\{1,2,3\},\{4\},\{5\},\{6\}} $ is nonzero.  Thus, $D$ cannot be a multiple of any non-Fakhruddin ray $E$.  

We consider the intersection $D \cdot F_{\{4,5,6\},\{1\},\{2\},\{3\}} $ first.  By Lemma \ref{n6 3111 lemma}, this intersection number is nonzero if inequalities (I-1)--(I-5) are satisfied.  We consider each of these in turn.  

The hypothesis $D \cdot F_{\{5,6\}} >0$ implies (I-1), as (I-1) is the analogue of inequality (J-4) or (J-5) for $D \cdot F_{\{5,6\}} $.  Similarly, $D \cdot F_{\{4,5\}}  >  0$ and 
$D \cdot F_{\{4,6\}}  >   0$ imply (I-2) and (I-3), respectively.  (I-4) follows from any of these nonzero intersection numbers, as $\Lambda \leq 2\ell$ implies $D$ is trivial.  

Thus only inequality (I-5) remains.  The hypotheses above do not guarantee that (I-5) is satisfied; hence, $D \cdot F_{\{4,5,6\},\{1\},\{2\},\{3\}} $ could be zero.  However, if it is, then we know $\lambda_1 + \lambda_2 + \lambda_3 \leq \ell+1$.

Let (I-1')-(I-5') denote the analogues for $D \cdot F_{\{1,2,3\},\{4\},\{5\},\{6\}} $ of the inequalities (I-1)-(I-5).  By an argument similar to that above, the hypotheses $D \cdot F_{\{1,2\}} >0$, $D \cdot F_{\{1,3\}} >0$, and $D \cdot F_{\{2,3\}} >0$ establish (I-1') through (I-4').  Once again, (I-5') remains; $D \cdot F_{\{1,2,3\},\{4\},\{5\},\{6\}} $ could be zero if $\lambda_4 + \lambda_5 + \lambda_6 \leq \ell+1$.

However, if both $\lambda_1 + \lambda_2 + \lambda_3 \leq \ell+1$ and $\lambda_4 + \lambda_5 + \lambda_6 \leq \ell+1$, we will get a contradiction.  We know $\Lambda$ is even.  If $\Lambda \leq 2\ell$, then $D$ is trivial, which contradicts the nonzero intersection numbers assumed as hypotheses.  If $\Lambda=2\ell+2$, then Fakhruddin shows  $D(\sL_2,\ell,\vec{\lambda})$ is a multiple of the pullback of the canonical ample line bundle on the GIT quotient $(\Pro^1)^n \dblq_{\vec{\lambda}} \SL_2$ (\cite{Fakh}*{Theorem 4.5}).  However, we  assumed that $E$ is not a GIT divisor.  Thus we must have either  $\lambda_1 + \lambda_2 + \lambda_3 > \ell+1$ or $\lambda_4 + \lambda_5 + \lambda_6 > \ell+1$ (or both), and hence at least one of the intersection numbers $D \cdot F_{\{4,5,6\},\{1\},\{2\},\{3\}} $ or $D \cdot F_{\{1,2,3\},\{4\},\{5\},\{6\}} $ is nonzero.

Next we consider Pattern II.  We will show that if $D:= D(\sL_2, \ell, \vec{\lambda})$ satisfies 
\begin{displaymath}
\begin{array}{cccl}
D \cdot F_{\{1,3\},\{2,4\},\{5\},\{6\}} & > &  0,\\
D \cdot F_{\{1,5\},\{3,4\},\{5\},\{6\}} & > &  0,\\
D \cdot F_{\{1,2,3\},\{4\},\{5\},\{6\}} & > &  0,
\end{array}
\end{displaymath}
then at least one of the intersection numbers $D \cdot F_{\{2,3,4\},\{1\},\{5\},\{6\}}  $ or $D \cdot F_{\{1,2\},\{3,4\},\{5\},\{6\}} $ is nonzero.  Thus, $D$ cannot be a multiple of any non-Fakhruddin ray $E$.  

We consider the intersection $D \cdot F_{\{2,3,4\},\{1\},\{5\},\{6\}} $ first.  By Lemma \ref{n6 3111 lemma}, this intersection number is nonzero if five inequalities are satisfied; we write these as (I-1'')-(I-5'').  We consider each of these in turn.  (I-1'') follows from $D \cdot F_{\{1,2,3\},\{4\},\{5\},\{6\}} >0$.  (I-2'') follows from  $D \cdot F_{\{1,5\},\{3,4\},\{5\},\{6\}} >0$.  (I-3'') follows from $D \cdot F_{\{1,3\},\{2,4\},\{5\},\{6\}} >0$.  (I-4'') follows from any of these intersection numbers being positive.  Thus only inequality (I-5'') remains.  This inequality is $\ell+2 \leq \lambda_1+\lambda_5+\lambda_6$.

If (I-5'') is satisfied, then $D \cdot F_{\{2,3,4\},\{1\},\{5\},\{6\}}  >0$, and we are done.  So suppose    $\ell+1 \geq \lambda_1+\lambda_5+\lambda_6$.  Then we will show that $D \cdot F_{\{1,2\},\{3,4\},\{5\},\{6\}} >0$ by verifying inequalities  (J-1)-(J-10) in Lemma \ref{n6 2211 lemma}.

The inequality (J-1) follows because $D$ is not trivial.  (J-2) follows because $\lambda_1 + \lambda_5 + \lambda_6 \leq \ell+1$ and $\lambda_i \geq 1 $ for all $i$, hence $\lambda_1 \leq \ell-1$.  (J-3) follows from $D \cdot F_{\{1,5\},\{3,4\},\{5\},\{6\}} >0$.  (J-4) follows from $D \cdot F_{\{1,2,3\},\{4\},\{5\},\{6\}} >0$.  (J-5) follows from $D \cdot F_{\{1,5\},\{3,4\},\{5\},\{6\}} >0$.  (J-6) follows from $D \cdot F_{\{1,3\},\{2,4\},\{5\},\{6\}}>0$.  (J-7) follows from $D \cdot F_{\{1,5\},\{3,4\},\{5\},\{6\}} >0$.  (J-8) follows from $D \cdot F_{\{1,2,3\},\{4\},\{5\},\{6\}} >0$ because we must have $\lambda_4+\lambda_5+\lambda_6 \geq \ell+1$ and $\lambda_3 \geq 1$.  Thus it remains to check inequalities (J-9) and (J-10).

Suppose that (J-9) fails, that is, $\lambda_1 + \lambda_2 \geq \ell-1 + \lambda_5 + \lambda_6$.  From $D \cdot F_{\{1,3\},\{2,4\},\{5\},\{6\}}>0$ we have that $\lambda_1 + \lambda_2 + \lambda_3 +\lambda_4 \leq 2\ell-2+\lambda_5 + \lambda_6$.  We may rewrite and combine these inequalities to obtain
\[  \lambda_1 + \lambda_2 + \lambda_3 +\lambda_4  - 2\ell+2 - \lambda_5 \leq \lambda_6 \leq \lambda_1 + \lambda_2 - \ell+1 -\lambda_5
\]
\[  \Longrightarrow \lambda_3 + \lambda_4 \leq \ell-1.
\]
But from $D \cdot F_{\{1,2,3\},\{4\},\{5\},\{6\}}>0$ we have $2 + \lambda_1 + \lambda_2 \leq \lambda_3 + \lambda_4 + \lambda_5 + \lambda_6$, and we can combine this with $\lambda_1 + \lambda_2 \geq \ell-1 + \lambda_5 + \lambda_6$ to obtain
\[ 2 + \lambda_1 + \lambda_2 - \lambda_3 - \lambda_4 - \lambda_5 \leq \lambda_6 \leq \lambda_1 + \lambda_2 - \ell + 1 - \lambda_5
\]
\[    \Longrightarrow \ell+1 \leq \lambda_3 + \lambda_4.
\]
The contradiction shows that (J-9) must be satisfied.  

Next we check (J-10).  Suppose that (J-10) fails, that is, $\lambda_3 + \lambda_4 \geq \ell-1 + \lambda_5 + \lambda_6$.  As above, we get $\lambda_1 + \lambda_2 \leq \ell-1$.  From  $D \cdot F_{\{1,5\},\{3,4\},\{2\},\{6\}}>0$ we have $2 + \lambda_3 + \lambda_4 \leq \lambda_1 + \lambda_2 + \lambda_5 + \lambda_6$, and combining this with the first inequality yields
\[ 2 + \lambda_3 + \lambda_4 - \lambda_1 - \lambda_2 - \lambda_5 \leq \lambda_6 \leq \lambda_1 + \lambda_2 - \ell + 1 - \lambda_5
\]
\[    \Longrightarrow \ell+1 \leq \lambda_1 + \lambda_2.
\]
The contradiction shows that (J-10) must be satisfied.  

Thus, since all ten inequalities (J-1)--(J-10) are satisfied, we may apply Lemma \ref{n6 2211 lemma} to conclude that $D \cdot F_{\{1,2\},\{3,4\},\{5\},\{6\}}>0$.

\end{proof}

\section{$\sL_2$ conformal blocks don't cover the symmetric nef cone for $n=9$} \label{dont cover 9 section}

The main result of this section is that for $n=9$, the divisor $B_2+B_3+2B_4$ is not a multiple of a symmetrized conformal block divisor for $\mathfrak{g}=\sL_2$, which we prove in Theorem \ref{no 112}.

\begin{lemma}  \label{lambda_i > 0} Suppose $D(\mathfrak{g},\ell,\vec{\lambda})$ is $S_n$-symmetric and $D \not\equiv 0$.  Then $\lambda_i > 0$ for all $i=1,\ldots,n$.  
\end{lemma}
\begin{proof}
Suppose that $\lambda_j = 0$.   Then we will show that $D \equiv 0$.  

Recall that Fakhruddin's intersection number with F-curves is 
\begin{displaymath} D \cdot F_{I_1,I_2,I_3,I_4} = \sum_{\vec{\mu} \in P_{\ell}^{4}} \deg \mathbb{V}_{\vec{\mu}} \, r_{\lambda_{I_1} \mu_{1}^{*}}  r_{\lambda_{I_2} \mu_{2}^{*}}  r_{\lambda_{I_3} \mu_{3}^{*}}  r_{\lambda_{I_4} \mu_{4}^{*}}.  
\end{displaymath}
Let $I_4 = \{\lambda_j\}$.  Then by fusion rules (see Prop. \ref{fusion rules}), to have $r_{\lambda_{I_4} \mu_{4}^{*}}\neq 0$, we must have $\mu_{4}^{*} = 0 = \mu_{4}$.  But then by propagation $\mathbb{V}_{\vec{\mu}}$ is a pullback from $\M_{0,3} \cong pt$, and hence $\deg \mathbb{V}_{\vec{\mu}} = 0$.  Thus $D \cdot F_{I_1,I_2,I_3,I_4} = 0$ whenever $I_4=\{\lambda_j\}$.

We assumed that $D(\vec{\lambda})$ is $S_n$-symmetric, and hence
$D\cdot F_{I_1,I_2,I_3,I_4} = 0$ whenever $\#I_4 =1$.  But it is
well-known that curves of the form $\{ F_{n-i,i,1,1} \}$ form a basis
of $H_2(\M_{0,n},\Q)^{S_n}$ (see for instance \cite{AGSS}*{Prop. 4.1}).  Hence $D \equiv 0$.    
\end{proof}

\begin{lemma} \label{112 extremal in nef cone} The divisor $E:=B_2 + B_3 + 2B_4$ spans an extremal ray of both $\Nef(\M_{0,9})$ and $\Nef(\M_{0,9})^{S_9}$.
\end{lemma}
\begin{proof}
We have $\rho = \dim \Pic(\M_{0,n}) = 2^{n-1} - \binom{n}{2} -1$ by
\cite{KeelIntersection}.  When $n=9$, we have $\rho = 219$.
Therefore, it suffices to find 218 curves that are independent in
homology and satisfy $E \cdot C=0$.  This can be done by computer.

First, make a list of all the F-curves $F_{I_1,I_2,I_3,I_4} $ such that $E \cdot F_{I_1,I_2,I_3,I_4} = 0$.  Then we can check that these curves form a family of rank 218 by forming the matrix of intersection numbers $D_J \cdot F_{I_1,I_2,I_3,I_4} $ and checking that this matrix has rank 218.  To make the calculation slightly more efficient, we can intersect only with a set of divisors $D_J$ forming a basis of $\Pic(\M_{0,n})$.

I have posted my \texttt{Macaulay2} code for this calculation on my website: \\ \neturltilde{http://www.math.uga.edu/~davids/dontcover/}{http://www.math.uga.edu/$\sim$davids/dontcover/}.
\end{proof}

\begin{definition} \label{def of I}
We define a system of inequalities $\mathcal{I}$ as follows:
\begin{enumerate}
\item[$(\mathcal{I}\mbox{-}1)$]  $\lambda_1 +\lambda_2 + \lambda_3 +\lambda_4 +\lambda_8 +\lambda_9 \leq 4\ell-2 +\lambda_5  +\lambda_6 +\lambda_7$
\item[$(\mathcal{I}\mbox{-}2)$] $\lambda_1 + \lambda_2 +\lambda_8 +\lambda_9 -(2\ell-2) \leq  \lambda_3 +\lambda_4 +\lambda_5+\lambda_6 +\lambda_7$
\item[$(\mathcal{I}\mbox{-}3)$] $ 2+\lambda_8 +\lambda_9  \leq  \lambda_1 + \lambda_2+ \lambda_3 +\lambda_4 +\lambda_5+\lambda_6 +\lambda_7$
\item[$(\mathcal{I}\mbox{-}4)$] $\ell+2 \leq \lambda_6 +\lambda_7+\lambda_8
  +\lambda_9$ 
\item[$(\mathcal{I}\mbox{-}5)$] $\lambda_1+ \lambda_2 + \lambda_3 +\lambda_4  -(2\ell-2) \leq  \lambda_5 +\lambda_6 +\lambda_7+\lambda_8 +\lambda_9$
\item[$(\mathcal{I}\mbox{-}6)$] $2 + \lambda_1 + \lambda_2 \leq \lambda_3 +\lambda_4 +\lambda_5+\lambda_6 +\lambda_7+\lambda_8 +\lambda_9$
\item[$(\mathcal{I}\mbox{-}7)$] $2\ell+2\leq  \Lambda = \lambda_1 + \lambda_2+ \lambda_3 +\lambda_4 +\lambda_5+\lambda_6 +\lambda_7+\lambda_8 +\lambda_9$
\item[$(\mathcal{I}\mbox{-}8)$] $2 + \lambda_1+\lambda_2+\lambda_3+\lambda_4 \leq 3\ell + \lambda_5+\lambda_6+\lambda_7$ 
\item[$(\mathcal{I}\mbox{-}9)$] $2+\lambda_1+\lambda_2 \leq \ell +\lambda_3+\lambda_4+ \lambda_5+\lambda_6+\lambda_7$ 
\item[$(\mathcal{I}\mbox{-}10)$]  $\ell+2 \leq \lambda_1 + \lambda_2+ \lambda_3 +\lambda_4 +\lambda_5+\lambda_6 +\lambda_7$ 
\end{enumerate}
\end{definition}

\begin{lemma} \label{I lemma} Suppose $\vec{\lambda}$ satisfies 
\begin{enumerate}
\item[i.] $\lambda_i \geq 1$ for $1\leq i \leq 5$
\item[ii.] $\lambda_i \leq \ell-1$ for $2 \leq i \leq 9$
\item[iii.] $2 \leq \lambda_i$ for $6 \leq i \leq 9$
\item[iv.] $\lambda_7 \geq \lambda_8 \geq \lambda_9$
\item[v.] $r_{\vec{\lambda}} \neq 0$.
\end{enumerate}
Then $D(\vec{\lambda}) \cdot F_{\{1,2,3,4,5\},\{6\},\{7\},\{8,9\}} > 0$ if and only if $\vec{\lambda}$ satisfies the system of inequalities $\mathcal{I}$. 
\end{lemma}
\begin{proof}
We use Proposition \ref{polytope Q}.

Since the sets $I_2$ and $I_3$ in the partition are singletons, the Rasmussen-Walton inequalities associated $r_{I_j \mu_j^{*}} \neq 0$ are just the fusion rules (see Prop. \ref{fusion rules}) for $j=2,3$.  In particular, to have $r_{\lambda_{6} \mu_{2}^{*}}\neq 0$, we must have $\mu_{2}^{*} = \lambda_6 = \mu_{2}$, and similarly $\mu_{3}^{*} = \lambda_7 = \mu_{3}$.

Write $\alpha= \mu_1$, $\beta=\mu_4$.  Thus, we only need to prove that $\mathcal{Q}$ contains a point of the form $(\alpha,\lambda_6,\lambda_7,\beta)$ where $\alpha \equiv \lambda_1+\cdots + \lambda_5 \pmod{2}$ and $\beta \equiv \lambda_8+\lambda_9 \pmod{2}$.  For this, we note that almost all the inequalities defining $\mathcal{Q}$ give the correct parity for $\mu$ when made equalities.  The exceptions are the conditions $0 \leq \alpha, \beta \leq \ell$.  Thus, by adjusting the upper and lower bounds for $\alpha, \beta$, we can arrange that if $\mathcal{Q}$ is
nonempty, it contains a point with the desired parity.  Specifically, if $\lambda_1 + \cdots + \lambda_5 \equiv 0 \pmod{2}$ we require $\alpha \geq 2$ else $\alpha \geq 1$.  If $\lambda_1 + \cdots +
\lambda_5 \equiv \ell \pmod{2}$ we require $\alpha \leq \ell$ else
$\alpha \leq \ell-1$.  Similarly if $\lambda_8 + \lambda_9 \equiv 0
\pmod{2}$ we require $\beta \geq 2$ else $\beta \geq 1$.  If $\lambda_8 + \lambda_9 \equiv \ell \pmod{2}$ we require $\beta\leq \ell$ else
$\beta \leq \ell-1$.  Note that if $\alpha$ and $\beta$ have the
desired parities, then since $\Lambda$ is even, $\alpha+\beta$ will
have the desired parity.  Let $\mathcal{Q}'$ be the polytope with these adjusted bounds.  
   
In the rest of the proof, I tediously explain how the inequalities defining
$\mathcal{I} $ are obtained from the inequalities defining
$\mathcal{Q}'$ by applying Fourier-Motzkin elimination to get rid of
$\alpha$ and $\beta$ and then discarding inequalities which 
follow from  the hypotheses above.

We consider the inequalities defining $\mathcal{Q}'$.  We have the four
parity-dependent upper and lower bounds for $\alpha$ and $\beta$.
By Lemma \ref{nonzero rank lemma}, the condition
$r_{\alpha\beta\lambda_{6}\lambda_{7}} \neq 0$ gives rise to eight Rasmussen-Walton 
inequalities (four with $|I|=1$, four with $|I|=3$).  By Lemma
\ref{boris lemma}, the condition $\deg
\mathbb{V}_{\alpha\beta\lambda_{6}\lambda_{7}} > 0$ gives rise to one
additional inequality: $\alpha + \beta + \lambda_6 + \lambda_7 \geq
2\ell+2$.  By Lemma \ref{nonzero rank lemma}, the condition $r_{\lambda_{1}\cdots\lambda_{5} \alpha^{*}}
\neq 0$ gives rise to 32 Rasmussen-Walton inequalities (six with $|I|=1$, 20 with
$|I|$=3, and six with $|I|$=5).  However, we may assume that
$\lambda_1 \geq \lambda_2 \geq \cdots \geq \lambda_5$; and then we
only need 6 of these 32 inequalities.  Finally, by Lemma \ref{nonzero
  rank lemma}, the condition $r_{\lambda_{8} \lambda_{9} \beta^{*}}
\neq 0$ gives rise to four Rasmussen-Walton inequalities (1 with $I = \emptyset$, three
with $|I|=2)$; but we may assume $\lambda_8 \geq \lambda_9$ and drop
one of these.  Thus, we get 22 inequalities: 7 lower
bounds for $\beta$, 7 upper bounds for $\beta$, and 8 which do not
involve $\beta$.  I have posted my notes for this calculation on my website:  \neturltilde{http://www.math.uga.edu/~davids/dontcover/}{http://www.math.uga.edu/$\sim$davids/dontcover/}. 

\begin{enumerate}
\item[($\beta$-LB 1)] $2 \ell +2 - \alpha -\lambda_6 - \lambda_7 \leq
  \beta$
\item[($\beta$-LB 2)] $-2 \ell + \alpha +\lambda_6 + \lambda_7 \leq
  \beta$
\item[($\beta$-LB 3)] $\alpha -\lambda_6 - \lambda_7 \leq
  \beta$
\item[($\beta$-LB 4)] $- \alpha +\lambda_6 - \lambda_7 \leq
  \beta$
\item[($\beta$-LB 5)] $- \alpha -\lambda_6 + \lambda_7 \leq
  \beta$
\item[($\beta$-LB 6)] $\lambda_8 - \lambda_9 \leq
  \beta$
\item[($\beta$-LB 7)] $\begin{array}{lc} 
\mbox{If $\lambda_8+\lambda_9 \equiv 0 \pmod{2}$: } & 2 \leq \beta \\
\mbox{If $\lambda_8+\lambda_9 \equiv 1 \pmod{2}$: } & 1 \leq \beta \\
\end{array}$ 
\item[($\beta$-UB 1)] $\beta \leq \alpha + 2 \ell - \lambda_6 - \lambda_7$
\item[($\beta$-UB 2)] $\beta \leq -\alpha + 2 \ell + \lambda_6 -
  \lambda_7$
\item[($\beta$-UB 3)] $\beta \leq -\alpha + 2 \ell - \lambda_6 +
  \lambda_7$
\item[($\beta$-UB 4)] $\beta \leq \alpha +  \lambda_6 +
  \lambda_7$
\item[($\beta$-UB 5)] $\beta \leq \lambda_8 + \lambda_9$
\item[($\beta$-UB 6)] $\beta \leq 2\ell - \lambda_8 -
  \lambda_9$
\item[($\beta$-UB 7)] $\begin{array}{lc} 
\mbox{If $\lambda_8+\lambda_9 \equiv \ell \pmod{2}$: } & \beta \leq \ell \\
\mbox{If $\lambda_8+\lambda_9 \equiv \ell-1 \pmod{2}$: } & \beta \leq \ell-1\\
\end{array}$ 
\end{enumerate}

We eliminate $\beta$ by testing each pair $(\mbox{$\beta$-LB $i$},
\mbox{$\beta$-UB $j$})$. Many of the inequalities so obtained follow
easily from our hypotheses, or from another bound.  For instance,
consider the first such pair, $(\mbox{$\beta$-LB 1},
\mbox{$\beta$-UB 1})$, created from the  first lower bound and first
upper bound: $2 \ell + 2 - \alpha - \lambda_6 - \lambda_7 \leq \alpha
+ 2 \ell - \lambda_6 - \lambda_7$ is equivalent to $\alpha \geq 1$, which was already on our list of inequalities for $\alpha$.  

After eliminating $\beta$ in this way, we are left with 11 lower bounds for $\alpha$ and 8
upper bounds for $\alpha$.  

\begin{enumerate}
\item[($\alpha$-LB 1)]  $2 - \lambda_6 -
  \lambda_7+ \lambda_8 + \lambda_9 \leq \alpha$
\item[($\alpha$-LB 2)]  $2\ell+2 - \lambda_6 -
  \lambda_7- \lambda_8 -\lambda_9  \leq \alpha$
\item[($\alpha$-LB 3)]  $\lambda_6 -
  \lambda_7- \lambda_8 -\lambda_9  \leq \alpha$
\item[($\alpha$-LB 4)]  $ -\lambda_6 +
  \lambda_7+\lambda_8 +\lambda_9 - 2 \ell \leq \alpha$
\item[($\alpha$-LB 5)]  $-\lambda_6 +
  \lambda_7- \lambda_8 -\lambda_9  \leq \alpha$
\item[($\alpha$-LB 6)]   $\lambda_6 +
  \lambda_7+ \lambda_8 -\lambda_9 - 2\ell  \leq \alpha$
\item[($\alpha$-LB 7)]
  $\lambda_1+\lambda_2+\lambda_3+\lambda_4+\lambda_5 - 4 \ell \leq \alpha$
\item[($\alpha$-LB 8)]   $\lambda_1+\lambda_2+\lambda_3-\lambda_4-\lambda_5 - 2 \ell \leq \alpha$
\item[($\alpha$-LB 9)]   $\lambda_1-\lambda_2-\lambda_3-\lambda_4-\lambda_5  \leq \alpha$
\item[($\alpha$-LB 10)]  $\begin{array}{lc} 
\mbox{If $\lambda_1+\cdots+\lambda_5 \equiv 0 \pmod{2}$: } & 2 \leq \alpha \\
\mbox{If $\lambda_1+\cdots+\lambda_5 \equiv 1 \pmod{2}$: } & 1 \leq \alpha\\
\end{array}$ 
\item[($\alpha$-LB 11)]  
$\begin{array}{lc} 
\mbox{If $\lambda_8+\lambda_9 \equiv \ell \pmod{2}$: } & \ell+2 -
\lambda_6-\lambda_7 \leq \alpha \\
\mbox{If $\lambda_8+\lambda_9 \equiv \ell-1  \pmod{2}$: } & \ell+3 -
\lambda_6-\lambda_7 \leq \alpha\\
\end{array}$ 
\item[($\alpha$-UB 1)]   $\alpha \leq 4 \ell -
  \lambda_6-\lambda_7-\lambda_8-\lambda_9$
\item[($\alpha$-UB 2)]   $\alpha \leq 2 \ell -
  \lambda_6-\lambda_7+\lambda_8+\lambda_9$
\item[($\alpha$-UB 3)]   $\alpha \leq 
  \lambda_6+\lambda_7+\lambda_8+\lambda_9$
\item[($\alpha$-UB 4)]   $\alpha \leq 2 \ell +
  \lambda_6-\lambda_7-\lambda_8+\lambda_9$
\item[($\alpha$-UB 5)]   $\alpha \leq 4 \ell - \lambda_1
 - \lambda_2-\lambda_3-\lambda_4+\lambda_5$
\item[($\alpha$-UB 6)]   $\alpha \leq 2 \ell -
  \lambda_1-\lambda_2+\lambda_3+\lambda_4+\lambda_5$
\item[($\alpha$-UB 7)]   $\alpha \leq 
  \lambda_1+\lambda_2+\lambda_3+\lambda_4+\lambda_5$
\item[($\alpha$-UB 8)]   $\begin{array}{lc} 
\mbox{If $\lambda_1+\cdots+\lambda_5 \equiv \ell \pmod{2}$: } & \alpha
\leq \ell \\
\mbox{If $\lambda_1+\cdots+\lambda_5 \equiv \ell-1 \pmod{2}$: } &
\alpha \leq \ell-1\\
\end{array}$ 
\end{enumerate}

We eliminate $\alpha$ by testing each pair $(\mbox{$\alpha$-LB $i$},
\mbox{$\alpha$-UB $j$})$. Many of the inequalities so obtained follow easily from our
hypotheses, or from another bound.  For instance, consider $(\mbox{$\alpha$-LB 3},
\mbox{$\alpha$-UB 5})$:
\begin{eqnarray*}
\lambda_6 -  \lambda_7- \lambda_8 -\lambda_9 & \leq &  4 \ell +
\lambda_1 - \lambda_2-\lambda_3-\lambda_4-\lambda_5 \\ 
\lambda_2 + \lambda_3 + \lambda_4 + \lambda_5 + \lambda_6 - 4 \ell
& \leq &  \lambda_1 + \lambda_7 + \lambda_8 + \lambda_9,
\end{eqnarray*}
and we recognize this last inequality as following from
$r_{\vec{\lambda}} \neq 0$ by Lemma \ref{nonzero rank lemma} with $I =
\{ 1,7,8,9\}$.  

We are nearly done.  We are left with ten inequalities coming from the pairs $(\mbox{$\alpha$-LB 1},\mbox{$\alpha$-UB 5})$, $(\mbox{$\alpha$-LB 1},
\mbox{$\alpha$-UB 6})$, $(\mbox{$\alpha$-LB 1}, \mbox{$\alpha$-UB
  7})$, $(\mbox{$\alpha$-LB 2}, \mbox{$\alpha$-UB 8})$, $(\mbox{$\alpha$-LB 2},
\mbox{$\alpha$-UB 5})$, $(\mbox{$\alpha$-LB 2}, \mbox{$\alpha$-UB 6})$, \\ $(\mbox{$\alpha$-LB 2},\mbox{$\alpha$-UB 7})$, $(\mbox{$\alpha$-LB 11},
\mbox{$\alpha$-UB 5})$, $(\mbox{$\alpha$-LB 11}, \mbox{$\alpha$-UB 6})$, and $(\mbox{$\alpha$-LB 11},
\mbox{$\alpha$-UB 7})$.

Finally, we note that the parity condition drops out.  For instance, consider the pair $(\mbox{$\alpha$-LB 2}, \mbox{$\alpha$-UB 8})$.  This gives 
\begin{displaymath}
\begin{array}{lc} 
\mbox{If $\lambda_1+\cdots+\lambda_5 \equiv \ell \pmod{2}$: } & \ell+2 \leq \lambda_6 +\lambda_7+\lambda_8
  +\lambda_9;\\
\mbox{If $\lambda_1+\cdots +\lambda_5 \equiv \ell -1\pmod{2}$: } &
\ell+3 \leq \lambda_6 +\lambda_7+\lambda_8 +\lambda_9.
\end{array}
\end{displaymath}

However, in the second line, if $\lambda_1+\lambda_2+\lambda_3+\lambda_4+\lambda_5 \equiv \ell-1 \pmod{2}$, then since $\Lambda \equiv 0 \pmod{2}$, we have $\lambda_6 +\lambda_7+\lambda_8 +\lambda_9 \equiv \ell-1\pmod{2}$ also.  But then the condition $\ell+2 \leq \lambda_6 +\lambda_7+\lambda_8 +\lambda_9$ implies that $\ell+3 \leq \lambda_6 +\lambda_7+\lambda_8 +\lambda_9$. So we may unify these two cases and simply write $\ell+2 \leq \lambda_6 +\lambda_7+\lambda_8 +\lambda_9$ for $(\mathcal{I}\mbox{-}4)$.

Thus we obtain the ten inequalities given in $\mathcal{I}$.  
\end{proof}

%
%

\begin{proposition} \label{main prop}
Suppose $\vec{\lambda}$ satisfies 
\begin{enumerate}
\item[i.] $\lambda_i \geq 1$ for $1\leq i \leq 5$
\item[ii.] $\lambda_i \leq \ell-1$ for $2 \leq i \leq 9$
\item[iii.] $2 \leq \lambda_i$ for $6 \leq i \leq 9$
\item[iv.] $\lambda_7 \geq \lambda_8 \geq \lambda_9$
\item[v.] $r_{\vec{\lambda}} \neq 0$.
\end{enumerate}
Write $D = D(\sL_2,\ell,\vec{\lambda})$.  Suppose that 
\begin{eqnarray*}
D \cdot F_{\{1,2,3,4,5,6\},\{7\},\{8\},\{9\}} & > &  0\\
D \cdot F_{\{1,2,3,4,8,9\},\{5\},\{6\},\{7\}} & > &  0
\end{eqnarray*}
Then $D \cdot F_{\{1,2,3,4,5\},\{6\},\{7\},\{8,9\}} > 0$.
\end{proposition}

\textit{Remark.}  We will see in the proof that there is nothing canonical about the set of F-curves used above.  Many other sets of F-curves of shape $6,1,1,1$ would force an F-curve of shape $5,2,1,1$ to have nonzero intersection with $D$.

\begin{proof}
We check the ten inequalities defining  $\mathcal{I}$.  


First we study consequences of $D \cdot F_{\{1,2,3,4,8,9\},\{5\},\{6\},\{7\}} >0$.  Recall Fakhruddin's formula for intersection numbers with F-curves, which is printed in equation (\ref{Fakh Fcurve formula}) above, and apply this to the curve $F_{\{1,2,3,4,8,9\},\{5\},\{6\},\{7\}}$.  The two point fusion rules say that for $j=2,3,4$ the ranks $r_{I_j \mu_j^*}$ are nonzero only when $\mu_{2} = \lambda_{5}$, $\mu_{3} = \lambda_{6}$, $\mu_{4} = \lambda_{7}$.  Thus, since $D \cdot
F_{\{1,2,3,4,8,9\},\{5\},\{6\},\{7\}}  >   0$ there exists $\mu_1 \in
P_{\ell}$ such that $\deg \mathbb{V}_{\mu_1,
  \lambda_{5},\lambda_{6},\lambda_{7}} > 0 $ and
$r_{\lambda_{1}\lambda_{2}
  \lambda_{3}\lambda_{4}\lambda_{8}\lambda_{9} \mu_1^{*} } \neq 0$.
The condition $\deg \mathbb{V}_{\mu_1,
  \lambda_{5},\lambda_{6},\lambda_{7}} > 0 $ requires that
$\mu_1+\lambda_5+\lambda_6+\lambda_7$ be even (so that $r_{\mu_1,
  \lambda_{5},\lambda_{6},\lambda_{7}} > 0$) and
$\mu_1+\lambda_5+\lambda_6+\lambda_7 > 2\ell$ (by Lemma \ref{boris
  lemma}), so we have $\mu_1+\lambda_5+\lambda_6+\lambda_7 \geq
2\ell+2$.

The condition $r_{\lambda_{1}\lambda_{2}
  \lambda_{3}\lambda_{4}\lambda_{8}\lambda_{9} \mu_1^{*} } \neq 0$
implies that $\mu_1 \leq 6\ell -\lambda_{1} - \lambda_2 -\lambda_3-\lambda_4-\lambda_8 -\lambda_9$ using Lemma \ref{nonzero rank lemma} with $I = \emptyset$.  Combining this with $2\ell+2-\lambda_5-\lambda_6-\lambda_7 \leq \mu_1$ yields the inequality $(\mathcal{I}\mbox{-}1)$:
\begin{align*}
2\ell+2-\lambda_5-\lambda_6-\lambda_7 \leq \mu_1 \leq 6\ell
-\lambda_{1} - \lambda_2 -\lambda_3-\lambda_4-\lambda_8 -\lambda_9\\
\lambda_{1} + \lambda_2+ \lambda_3 +\lambda_4 +\lambda_8 +\lambda_9 \leq 4\ell-2 +
\lambda_5 +\lambda_6 +\lambda_7.
\end{align*}

The condition $r_{\lambda_{1}\lambda_{2}  \lambda_{3}\lambda_{4}\lambda_{8}\lambda_{9} \mu_1^{*} } \neq 0$
implies that $\mu_1 \leq 4\ell -\lambda_{1} - \lambda_2 -\lambda_8 -\lambda_9 +\lambda_3+\lambda_4$ using Lemma \ref{nonzero rank lemma} with $I = \{3,4\}$.  Combining this with $2\ell+2-\lambda_5-\lambda_6-\lambda_7 \leq \mu_1$ yields the inequality $(\mathcal{I}\mbox{-}2)$.

The condition $r_{\lambda_{1}\lambda_{2}  \lambda_{3}\lambda_{4}\lambda_{8}\lambda_{9} \mu_1^{*} } \neq 0$
implies that $\mu_1 \leq 2\ell -\lambda_8 -\lambda_9 +\lambda_{1} + \lambda_2 +\lambda_3+\lambda_4$ using Lemma \ref{nonzero rank lemma} with $I = \{1,2,3,4\}$.  Combining this with $2\ell+2-\lambda_5-\lambda_6-\lambda_7 \leq \mu_1$ yields the inequality $(\mathcal{I}\mbox{-}3)$.

We now proceed to check inequalities $(\mathcal{I}\mbox{-}8)$, $(\mathcal{I}\mbox{-}9)$, and $(\mathcal{I}\mbox{-}10)$.  
 We argued above that since $\deg \mathbb{V}_{\mu_1,\lambda_5,\lambda_6,\lambda_7} >0$ we have $\mu_1+\lambda_5+\lambda_6+\lambda_7 \geq 2\ell+2$.  But $\deg \mathbb{V}_{\mu_1,\lambda_5,\lambda_6,\lambda_7} >0$ also requires $r_{\mu_1,\lambda_5,\lambda_6,\lambda_7} >0$, and this gives $\mu_1 \leq \lambda_5+\lambda_6+\lambda_7$ by Lemma \ref{nonzero rank lemma} with $I=\{2,3,4\}$.  Combining these inequalities, we obtain $\lambda_5+\lambda_6+\lambda_7 \geq \ell+1$.  We also apply the hypothesis that $\lambda_i \leq \ell-1$ for $i \geq 2$.  The desired inequalities then follow.

Next we study consequences of $D \cdot F_{\{1,2,3,4,8,9\},\{5\},\{6\},\{7\}} >0$.  By an argument similar to that used above, we get 
$\mu_1+\lambda_7+\lambda_8+\lambda_9 \geq 2\ell+2$.  We also have $r_{\lambda_{1} \lambda_{2}\lambda_{3}\lambda_{4}\lambda_{5}\lambda_{6} \mu_1^{*} } \neq 0$, and by  Lemma \ref{nonzero rank lemma} with $I = \{5,6\}$ this implies $\mu_1 \leq 4 \ell -\lambda_1 -\lambda_2 -\lambda_3-\lambda_4 +\lambda_5 + \lambda_6$. Combining these inequalities yields  inequality $(\mathcal{I}\mbox{-}5)$.  

The condition $r_{\lambda_{1} \lambda_{2}\lambda_{3}\lambda_{4}\lambda_{5}\lambda_{6} \mu_1^{*} } \neq 0$ implies that $\mu_1 \leq 2\ell -\lambda_1 -\lambda_2 +\lambda_{3} + \lambda_4 +\lambda_5+\lambda_6$ using Lemma \ref{nonzero rank lemma} with $I=\{3,4,5,6\}$.  Combining this with $\mu_1+\lambda_7+\lambda_8+\lambda_9 \geq 2\ell+2$ yields   the inequality $(\mathcal{I}\mbox{-}6)$.

The condition $r_{\lambda_{1} \lambda_{2}\lambda_{3}\lambda_{4}\lambda_{5}\lambda_{6} \mu_1^{*} } \neq 0$ implies that $\mu_1 \leq  \lambda_1 +\lambda_2 +\lambda_{3} + \lambda_4 +\lambda_5+\lambda_6$ using Lemma \ref{nonzero rank lemma} with $I=\{1,2,3,4,5,6\}$.  Combining this with $\mu_1+\lambda_7+\lambda_8+\lambda_9 \geq 2\ell+2$ yields   the inequality $(\mathcal{I}\mbox{-}7)$.

It remains only to check the inequality $(\mathcal{I}\mbox{-}4)$.  
The proof begins with an argument that similar to that used to prove the first inequality.  We get $\mu_1+\lambda_7+\lambda_8+\lambda_9 \geq 2\ell+2$.  The condition $\deg \mathbb{V}_{\vec{\mu}} > 0$ also requires $r_{\vec{\mu}} > 0$, and by  Lemma \ref{nonzero rank lemma} with $I = \{7,8,9\}$ this implies $\mu_1 \leq  \lambda_7 +\lambda_8 +\lambda_9$.  Combining these inequalities yields $\ell+1 \leq \lambda_7 +\lambda_8 +\lambda_9$.  Combining this with $\lambda_6 \geq 2$ yields the desired result.

We have verified all ten inequalities comprising the system of inequalities $\mathcal{I}$.  Thus, we may apply Lemma \ref{I lemma} to conclude that $D \cdot F_{\{1,2,3,4,5\},\{6\},\{7\},\{8,9\}} > 0$.

\end{proof}

\begin{lemma} \label{one ell lemma} Suppose $D=D(\sL_2,\ell, \vec{\lambda})$ is an $S_n$-symmetric conformal block divisor.
\begin{enumerate}
\item If $\lambda_i=\lambda_j=\ell$ for some $i\neq j$, then $D
  \cdot F_{I_1,I_2,I_3,I_4} = 0$ for any partition with $I_4 =\{\lambda_i,\lambda_j\}$.
\item Suppose $\ell \geq 4$. If $\lambda_i=\lambda_j=\lambda_k = 1$ where
  $i,j,k$ are pairwise distinct, then \\$D
  \cdot F_{\{ijk\}^{c},\{i\},\{j\},\{k\}} = 0$. 
\end{enumerate}
\end{lemma}
\begin{proof}
For the first statement: combining Fakhruddin's intersection number formula
(\ref{Fakh Fcurve formula}) and the three point fusion rules, we see that if $I_4
=\{\lambda_i,\lambda_j\}$, then $ r_{\lambda_{I_4} \mu_{4}^{*}} = 0$
unless $\mu_4=0$.  But then $\mathbb{V}_{\vec{\mu}} $ is a
pullback from $\M_{0,3} \cong pt$, and hence trivial.  Thus every term
in the intersection number formula is zero.

For the second statement: in the intersection number formula printed above,
we see that to have $r_{\lambda_{I_2} \mu_{2}^{*}}  r_{\lambda_{I_3}
  \mu_{3}^{*}}  r_{\lambda_{I_4} \mu_{4}^{*}} \neq 0$ we must have
$\vec{\mu} = (\mu_1,1,1,1)$.  However, by Lemma \ref{nonzero rank
  lemma}, then to have $r_{\vec{\mu}} \neq 0$, we must have $\mu_1 \in
\{1,3\}$; but then by Lemma \ref{boris lemma} $\mathbb{V}_{\vec{\mu}}
=0$ if $\ell \geq 4$. 
\end{proof}

\begin{theorem} \label{no 112}
The divisor $E:=B_2+B_3+2B_4$ is not a multiple of a symmetrized conformal block divisor for $\mathfrak{g}=\sL_2$.
\end{theorem}
\begin{proof}
By Lemma \ref{112 extremal in nef cone}, we know that $E$ is extremal in the nef cone of $\M_{0,9}$.  Hence, if it is a multiple of the symmetrized conformal block $\sum_{\sigma \in S_n} D(\sigma \vec{\lambda}$), it is already a multiple of $D(\vec{\lambda})$.  Thus, it is enough to show that $E$ is not a multiple of $D(\vec{\lambda})$ for any $\vec{\lambda}$.  

We can compute all the conformal blocks for $n=9$ points,
$\mathfrak{g}=\sL_2$, and $1 \leq \ell \leq 3$; none of
these give $E$.  Thus, we may assume that $\ell \geq 4$.  

We may also assume that $\lambda_i \leq \ell-1$ for $i \geq 2$, and $\lambda_i \geq 2$ for $i \geq 6$, as we now explain.  
We know that $E \cdot F_{I_1,I_2,I_3,I_4} \neq 0$ for any F-curve of shape
$3,2,2,2$; thus, by Lemma \ref{one ell lemma} above, we may assume
that $\lambda_h = \ell$ for at most one $h \in \{1,\ldots,n\}$.  We
also know that $E \cdot F_{I_1,I_2,I_3,I_4} \neq 0$ for any F-curve of shape
$6,1,1,1$; thus, by Lemma \ref{one ell lemma} above, we may assume
that $\lambda_i = \lambda_j= 1$ for at most two values $i,j \in
\{1,\ldots,n\}$.  By permuting the weights if necessary, we may assume
that $h=1$ and $i,j \in \{1,\ldots, 5\}$, and hence that $2 \leq
\lambda_6,\ldots,\lambda_9 \leq \ell-1$.

Reordering the weights in $\vec{\lambda}$ if necessary, we may
furthermore assume that $\lambda_7 \geq \lambda_8 \geq \lambda_9$.

We know $E \cdot F_{I_1,I_2,I_3,I_4}  > 0$ for any F-curve of shape $6,1,1,1$.  Thus, in particular, $D$ has positive intersection with the two curves listed in the statement of Proposition \ref{main prop}.  Applying Proposition \ref{main prop}, we must have $D \cdot F_{\{1,2,3,4,5\},\{6\},\{7\},\{8,9\}} > 0$.  But then $D$ cannot be a multiple of $E$, since $E\cdot F_{I_1,I_2,I_3,I_4}  = 0$ for any F-curve of shape $5,2,1,1$.

\end{proof}

\section{Open questions}

Here are some open questions about conformal block determinants.  

The following question is probably the most important and also the furthest out of reach:
\begin{problem} 
Give an algorithm that does the following:\\
\mbox{} \hspace{1in} \textsc{Input}: a nef divisor $D$\\
\mbox{} \hspace{1in} \textsc{Output}: either ``$D$ is not a conformal block bundle,'' \\
\mbox{} \hspace{1in} or else $(\mathfrak{g}, \ell, \vec{\lambda})$ such that $D$ is a multiple of $c_{1} \mathbb{V}(\mathfrak{g},\ell,\vec{\lambda})$. 
\end{problem}

\begin{problem}
\begin{enumerate}
\item[(ia)] Find the analogue of the Rasmussen-Walton inequalities for general $\mathfrak{g}$; i.e. show that $r_{\vec{\lambda}} \neq 0$ is equivalent to the existence of a lattice point in a polytope.
\item[(ib)]  Show that $r_{\vec{\lambda}}$ counts the number of lattice points in a polytope.
\item[(iia)]  Find the analogue of the Proposition \ref{polytope Q} for general $\mathfrak{g}$; i.e. show that $D(\mathfrak{g},\ell,\vec{\lambda}) \cdot F_{I_1,I_2,I_3,I_4} \neq 0$ is equivalent to the existence of a point of specified parity in a polytope.
\item[(iib)] Does $D(\mathfrak{g},\ell,\vec{\lambda}) \cdot F_{I_1,I_2,I_3,I_4}$ count the number of points of specified parity in a polytope?
\end{enumerate}
\end{problem}
It is well-known that ranks of $\sL_2$ and $\sL_3$ conformal block bundles count the number of lattice points in a polytope \cite{RasmussenWalton}.  (We did not use the full power of Rasmussen and Walton's results in this paper.)  It is widely conjectured that the fusion rules (ranks of conformal block bundles on $\M_{0,3}$) count the number of lattice points in a polytope, but at the time of this writing, this has only proven for $\sL_2$ and $\sL_3$.  With this in hand, it seems likely that one could write down the analogues of the Rasmussen-Walton inequalities for more general $\mathfrak{g}$.  

It also seems likely that one might be able to generalize Proposition \ref{polytope Q} to more general $\mathfrak{g}$ (this is (iia) above).   It's not clear, however, that the resulting polytopes would be sufficiently manageable to extend the results of this paper to other Lie algebras. Finally, (ib) and (iia) together make (iib) a natural question.

\begin{problem}
Find good techniques for showing  $\deg \mathbb{V}_{\mu_1 \mu_2 \mu_3 \mu_4} \neq 0$ for a general $\mathfrak{g}$. 
\end{problem}
To me, it is difficult to see when  Fakhruddin's formula \cite{Fakh}*{Cor. 3.5} for $\deg \mathbb{V}_{\mu_1 \mu_2 \mu_3 \mu_4}$ is nonzero.  A formula with all nonnegative summands is desirable, because then we could stop after finding one positive term.  Alternatively, it seems possible that there might be a combinatorial characterization of when a set of four weights $\vec{\mu}$ gives a trivial conformal block bundle on $\Pro^{1}$.

\begin{problem}
If $D(\sL_2,\ell,\vec{\lambda})$ is $S_n$-symmetric and nontrivial, then is the set of weights $\vec{\lambda}$ also $S_n$-symmetric?  Does a similar statement hold for other $\mathfrak{g}$ if we also take into account the involution $*$?
\end{problem}
If the answer to the above question is yes, then it is possible to give a significantly simpler proof of Theorem \ref{no 112} than that given in Section \ref{dont cover 9 section} above.

\begin{problem}
For a fixed $\mathfrak{g}$, is the cone of conformal block determinants finitely generated?  That is, given $\mathfrak{g}$, does there exist $\ell_{0}$ such that 
\[  \{ D(\mathfrak{g},\ell,\vec{\lambda}) \mid \ell \in \N\} \subseteq \operatorname{ConvHull} \{ mD(\mathfrak{g},\ell,\vec{\lambda}) \mid \ell \in \N, m \in \R_{\geq 0}, \ell \leq \ell_0\} ?
\]

\end{problem}

\newcounter{lastbib}

\section*{Software Packages Referenced}

\end{document}